\documentclass[oneside,reqno,10pt]{amsart}
\usepackage{graphicx,amsthm,amssymb,amsmath,tikz,amsfonts} 
\usepackage{mathtools}
\usepackage{booktabs,caption}
\usepackage{bm}
\usepackage{enumitem}
\usepackage{float}
\usepackage{siunitx}

\usepackage[a4paper, left=3.3cm, right=3.3cm, top=3cm, bottom=3cm]{geometry}

\usepackage[symbol]{footmisc}

\usepackage{hyperref}
\hypersetup{
    colorlinks=true,
    linkcolor=black,
    citecolor=black,
    filecolor=black,
    urlcolor=blue,
}
\numberwithin{equation}{section}

\newcommand{\RNum}[1]{\uppercase\expandafter{\romannumeral #1\relax}}

\theoremstyle{plain}
\newtheorem{theorem}{Theorem}[section]

\newtheorem{lemma}{Lemma}

\theoremstyle{definition}

\newtheorem{defn}{Definition}[section]

\newcommand{\bburl}[1]{\textcolor{blue}{\url{#1}}}

\numberwithin{equation}{section}

\begin{document}
\author{Metin Can Aydemir}
\address{Department of Mathematics, Bilkent University, 06800 Ankara, Türkiye }
\email{\bburl{can.aydemir@bilkent.edu.tr}}
\author{Muhammet Boran}
\address{Department of Mathematics, Bilkent University, 06800 Ankara, Türkiye }
\email{\bburl{muhammet.boran@bilkent.edu.tr}}
\email{\bburl{boranmet48@gmail.com}}

\title{Improved Averaged Distribution of $d_3(n)$ in Prime Arithmetic Progressions}

\begin{abstract}
We say that $d_3(n)$ has exponent of distribution $\theta$ if, for every $\varepsilon>0$, the expected asymptotic holds uniformly for all moduli $q \le x^{\theta-\varepsilon}$. Nguyen proved, following earlier work of Banks, Heath-Brown, and Shparlinski, that after averaging over reduced residue classes $a \bmod q$, the function $d_3(n)$ has exponent of distribution $2/3$. Using the Petrow--Young subconvexity bound for Dirichlet $L$-functions, we improve this to $8/11$ when averaging over residue classes modulo a prime $q$.
\end{abstract}

\maketitle
\noindent

\section{Introduction}
\subsection{Historical Context}
The study of the distribution of divisor functions in arithmetic progressions has a rich history, with significant breakthroughs occurring in recent decades.
\begin{defn}
    For any $n,\ k \in\mathbb{Z}^{+}$, the \emph{$k$-fold divisor function} $d_k(n)$ counts the number of ways to write $n$ as a product of $k$ positive integers,
    \begin{equation*}
        d_k(n)\ :=\ \sum_{n_1n_2\dots n_k=n} 1,
    \end{equation*}
    where each $n_i \in \mathbb{Z}^+$.
\end{defn}
\begin{defn}
    Let $n$ be a fixed positive integer and consider the sum of the $n$th powers of the primitive $q$th roots of unity. This sum is known as \emph{Ramanujan's sum} and is denoted by
    \begin{equation*}
        c_q(n)\ :=\ \sideset{}{'}\sum_{m\ (q)} e\Big( \frac{mn}{q} \Big),
    \end{equation*}
    where $e(z) := \exp{(2\pi i z)}$ for any $z \in \mathbb{R}$, and $\sideset{}{'}\sum\limits_{m\ (q)}$ indicates that $m$ runs over residues coprime to $q$.
\end{defn}
  It is well known that this sum reduces to the Möbius function when $n$ is coprime to $q$; in that case, $c_q(n)=\mu(q)$.
  
Classical unpublished work of Selberg, Linnik, and Hooley from the 1950s established that the usual divisor function $d(n)$ has exponent of distribution $2/3$. For the ternary divisor function $d_3(n)$, progress beyond $1/2$ was much more difficult. Friedlander and Iwaniec \cite{FI} obtained the first improvement, proving an exponent of distribution $1/2+1/230$. This was subsequently improved to $1/2+1/82$ by Heath-Brown \cite{HB}, and to $1/2+1/46$ by Fouvry, Kowalski, and Michel \cite{FKM} (in the case of prime moduli).

A different line of development used averaging over residue classes to extend the range of distribution. In 2005, Banks, Heath-Brown, and Shparlinski \cite{BHS} showed that by averaging over all coprime residue classes modulo $q$, one can obtain a power-saving error term for divisor sums for the full range of moduli $q$. In particular, Nguyen \cite{Ng} applied the method of Banks, Heath-Brown, and Shparlinski to the ternary divisor function and obtained an averaged exponent of distribution $2/3$ for $d_3(n)$. In concrete terms, Nguyen showed that when one averages over all coprime residue classes $a\  (\bmod\  q)$, the error term in the sum of $d_3(n)$ up to $X$ over the progression $n\equiv a\  (\bmod\  q)$ exhibits cancellation for all moduli $q\ll X^{2/3-\varepsilon}$.This result matched the
$2/3$ distribution level on average, corresponding to the classical Selberg–Hooley bound in an averaged setting. Blomer \cite{Bl} introduced a simplified, more analytic approach using the Voronoi summation formula, which streamlined the analysis of divisor sums in progressions.
 
Parry \cite{Pa1} revisited the distribution of $d_3(n)$ in arithmetic progressions and gave a proof, similar to Blomer’s, that $d_3$ has exponent of distribution $2/3$ on average.

Recently, in the case $k=4$, Parry \cite{Pa2} used the Petrow-Young \cite{PY} subconvexity bound for Dirichlet $L$-functions to show that $d_4(n)$ has exponent of distribution $4/7$ after averaging over $a\  (\bmod q)$, for prime $q$.

\subsection{Statement of Main Results}
 In this paper, we improve the averaged level of distribution for the ternary divisor function $d_3(n)$  to $8/11$ by using the Petrow–Young \cite{PY} subconvexity bound. This work was motivated by a problem originally suggested to us by Parry.
 
Following the method of Parry \cite{Pa2}, we extend Parry’s result \cite{Pa1} for $d_3(n)$, which gave exponent of distribution $2/3$, to exponent $8/11$ in the case of prime moduli. We thus obtain the following theorem.
\begin{theorem}\label{maintheorem}
    Let 
    \begin{align*}
        E_{a/q}(s)\ &=\ \sum_{n=1}^{\infty} \frac{d_3(n)}{n^s} e\left( \frac{na}{q} \right), \\
        f_{a/q}(x)\ &=\ \underset{s=1}{\operatorname{Res}} \left\{ \frac{E_{a/q}(s) \left( (2x)^s - x^s \right)}{s} \right\}, \\
        \Delta(a/q)\ &=\ \sum_{x \leq n \leq 2x} d_3(n) e\left( \frac{na}{q} \right)\ -\ f_{a/q}(x).
    \end{align*}
    Then for a prime $q \le x^{8/11}$ and any $\varepsilon > 0$,
    \begin{equation*}
        \sum_{a=1}^q |\Delta(a/q)|^2\ \ll_{\varepsilon}\ x^{3/2+\varepsilon} q^{11/16}.
    \end{equation*}
\end{theorem}
 Let
\begin{equation*}
    \theta_{n}(q)=\frac{\mu(q/(q,n))}{\phi(q/(q,n))}.
\end{equation*}
 It is straightforward to check that
\begin{equation*}
    \sum_{\begin{subarray}{c}n=1\\ d|n\end{subarray}}^{\infty}\frac{d_{3}(n)\chi(n/d)}{n^{s}}
\end{equation*}
is holomorphic past $s=1$ for non-principal Dirichlet characters. Thus, after decomposing the sum according to the residue class $r\ (\bmod q)$, then according to $d=(r,q)$, and finally using Dirichlet characters, we obtain
\begin{equation*}
f_{a/q}(x) = \mathrm{Res}_{s=1}\left\{\frac{(2x)^{s}-x^{s}}{s}\sum_{n=1}^{\infty}\frac{d_{3}(n)\theta_{n}(q/(q,a))}{n^{s}}\right\}=:F_{x}\left(\frac{q}{(q,a)}\right).
\end{equation*}
Then, letting
\begin{equation*}
\mathcal{M}_{x}(q,a)=\frac{1}{q}\sum_{d|q}c_{d}(a)F_{x}(d)\qquad\quad E_{x}(q,a)= \sum_{\begin{subarray}{c}x<n\leq 2x\\ n\equiv a(q)\end{subarray}}d_{3}(n)-\mathcal{M}_{x}(q,a)
\end{equation*}
 we get
\begin{equation*}
\sum_{N=1}^{q}\mathcal{M}_{x}(q,N)e\left(\frac{Nr}{q}\right)=F_{x}\left(\frac{q}{(q,r)}\right)
\end{equation*}
 so
\begin{align*}
E_{x}(q,a)&=\frac{1}{q}\sum_{r=1}^{q}e\left(-\frac{ar}{q}\right)\left(\sum_{x<n\leq 2x}d_{3}(n)e\left(\frac{nr}{q}\right)-\sum_{N=1}^{q}\mathcal{M}_{x}(q,N)e\left(\frac{Nr}{q}\right)\right)\\
&=\frac{1}{q}\sum_{r=1}^{q}e\left(-\frac{ar}{q}\right)\Delta(r/q)
\end{align*}
 and therefore
\begin{equation}\label{equidistrubution}
    \sum_{a=1}^{q}|E_{x}(q,a)|^{2} = \frac{1}{q}\sum_{r=1}^{q}|\Delta(r/q)|^{2}.
\end{equation}
Applying the Cauchy–Schwarz inequality to the above identity, we obtain the following theorem.
\begin{theorem}
    For prime $q\leq x^{8/11}$
    \begin{equation}
        \sum_{a=1}^{q}|E_{x}(q,a)|\ \ll_{\varepsilon}\ x^{3/4+\varepsilon} q^{11/32} ,
    \end{equation}
    which shows equidistribution on average up to $q\leq x^{8/11}$. 
\end{theorem}
In Section \ref{sec:lemmas} we collect the auxiliary results needed for the proof of Theorem \ref{maintheorem}. A key ingredient is Lemma \ref{bound of L^3}, where we establish bounds related to the third moment of Dirichlet $L$-functions using the Petrow–Young \cite{PY} subconvexity result. These bounds are then applied in Lemmas \ref{lemma6} and \ref{lemma9} to obtain cancellation in sums involving Ramanujan sums, which are crucial in our analysis. Finally, in Section \ref{proof-of-theorem}, we combine these ingredients to complete the proof of Theorem \ref{maintheorem}.
\section{Preliminaries}\label{sec:lemmas}
\begin{lemma}\label{lemma1}
    Let $E_{h/q}(s)$ be as in Theorem \ref{maintheorem} and for $(h, q) = 1$, let $w: [0,\infty) \rightarrow \mathbb{R}$ be smooth and of compact support, and let
\begin{align*}
     \tilde{\Delta}(h/q) \ &=\  \sum_{n=1}^{\infty} d_3(n) e\left( \frac{nh}{q} \right) w(n) - \underset{s=1}{\operatorname{Res}} \left\{ E_{h/q}(s) \int_0^\infty w(t)\, t^{s-1} \,dt \right\},\\
      U(x)&=\frac{1}{2\pi i} \int_{(c)}\left( \frac{\Gamma(s/2)}{\Gamma((1 - s)/2)} \right)^3 \frac{ds}{x^s}, \quad \text{for } x > 0 \text{ and } 0 < c < \frac{1}{6},\\
      V(x)&= \frac{1}{2\pi i} \int_{(c)}\left(\frac{\Gamma((s+1)/2)}{\Gamma((2-s)/2)}\right)^3 \frac{ds}{x^s}, \quad \text{for } x>0 \text{ and } 0 <c < \frac{1}{6}, \\
      \hat{w}_q(n)\  &=\  \int_0^\infty w(t)\, U(Nt)\, dt, \ \ \hat{\nu}_q(n) =\int_0^\infty w(t)\, V(Nt)\, dt,\quad \text{with}\ \ \  N \ =\  \frac{\pi^3 n}{q^3},\\
      A_{h/q}(n) \ &=\  \frac{1}{2} \sum_{abc = n} \underbrace{\sum_{x,y,z=1}^q \Biggl(e\left( \frac{ax+by+cz+hxyz}{q} \right) + e\left( \frac{ax+by+cz-hxyz}{q} \right)\Biggr)}_{=: R^{+}_{a,b,c}(h/q)},\\
       B_{h/q}(n) \ &=\  \frac{1}{2} \sum_{abc = n} \underbrace{\sum_{x,y,z=1}^q \Biggl(e\left( \frac{ax+by+cz+hxyz}{q} \right) - e\left( \frac{ax+by+cz-hxyz}{q} \right)\Biggr)}_{=:R^{-}_{a,b,c}(h/q)}.
\end{align*}
Then,
    \begin{equation*}
        \tilde{\Delta}(h/q) = \frac{\pi^{3/2}}{q^3} \sum_{n=1}^{\infty} A_{h/q}(n)\, \hat{w}_q(n) + \frac{i^9 \pi^{3/2}}{q^3} \sum_{n=1}^{\infty} B_{h/q}(n)\hat{\nu}(q).
    \end{equation*}
\end{lemma}
\begin{proof}
    This is Theorem 2 of Ivi\'c's paper \cite{Iv}.
\end{proof}
The following Lemma concerns the bound for $A_{h/q}(n)$ from Ivić-Voronoi formula.
\begin{lemma}\label{lemma2}
    Let $A_{h/q}(n)$ and $B_{h/q}(n)$ be as in Lemma \ref{lemma1} with prime $q$ and $(h,q)=1$.  If $q\mid n$, then
    \begin{equation*}
        A_{h/q}(n),B_{h/q}(n)\;\ll_{\varepsilon}\;n^{\varepsilon}\,(q^2,n).
    \end{equation*} for any $\varepsilon>0$. Moreover, if $q\nmid m$, then
    \begin{equation*}
        \sum_{h=1}^{q-1}A_{h/q}(n)\,\overline{A_{h/q}(m)}\;\ll_{\varepsilon}\;(mn)^{\varepsilon}q(q^2,n),
    \end{equation*} for any $\varepsilon>0$.
\end{lemma}
\begin{proof}
We prove the claim only for $A_{h/q}(n)$, since the proof for $ B_{h/q}(n)$ is identical. By definition, 
    \begin{equation*}
        A_{h/q}(n)=\frac{1}{2}\sum_{abc=n}\sum_{x,y,z=1}^{q}\left(e\left(\frac{ax+by+cz+hxyz}{q}\right)+e\left(\frac{ax+by+cz-hxyz}{q}\right)\right).
\end{equation*}
    Since $q\mid n$, at least one of $a,b,c$ must be divisible by $q$. First, assume that $q\mid c$. We have
\begin{multline*}
    R^+_{a,b,c}(h/q)=\sum_{x,y=1}^{q}e\left(\frac{ax+by}{q}\right)\sum_{z=1}^{q}e\left(\frac{(c-hxy)z}{q}\right)+\\
    \sum_{x,y=1}^{q}e\left(\frac{ax+by}{q}\right)\sum_{z=1}^{q}e\left(\frac{(c+hxy)z}{q}\right).
\end{multline*}
In the first double sum, if $c\not\equiv hxy\pmod{q}$, then it is well-known that the inner sum is $0$. Thus, the only relevant case is $xy\equiv c\bar h\equiv 0\pmod{q}$. Same situation is also valid for the second double sum. Therefore, 
\begin{align*}
        R^+_{a,b,c}(h/q)=2q\sum_{\underset{q\mid xy}{x,y=1}}^{q}e\left(\frac{ax+by}{q}\right)&=2q\left(1+\sum_{x=1}^{q-1}e\left(\frac{ax}{q}\right)+\sum_{y=1}^{q-1}e\left(\frac{by}{q}\right)\right)\\
        &=2q(1+c_q(a)+c_q(b))\\
        &=2q(2+c_q(a)+c_q(b)+c_q(c)-q).
\end{align*} 
By symmetry, the same formula holds whenever $q\mid a$ or $q\mid b$, hence for all triples with $q\mid abc$. We have
\begin{equation*}
        R^+_{a,b,c}(h/q)=2q(2+c_q(a)+c_q(b)+c_q(c)-q)=2k_{a,b,c}q^2-2(q^2+q)
\end{equation*} where $q^{k_{a,b,c}}=(q,a)(q,b)(q,c)$, or equivalently, $k_{a,b,c}$ equals the number of elements among $a,b,c$ divisible by $q$. Therefore, 
\begin{align*}
        A_{h/q}(n)&=\frac{1}{2}\sum_{abc=n}R^+_{a,b,c}(h/q)\\
        &=q^2\sum_{abc=n} k_{a,b,c}-(q^2+q)\sum_{abc=n}1\\
        &=q^2\log_q\left[\prod_{abc=n} (q,a)(q,b)(q,c)\right]-(q^2+q)d_3(n).
\end{align*} Let $v_q(n)=\alpha\geq 1$ and $n=q^\alpha \ell$. Then, we make the change of variables $(a,b,c)=(q^xa',q^yb',q^zc')$ where $(q,a')=(q,b')=(q,c')=1$.
\begin{align*}
        \prod_{abc=n} (q,a)(q,b)(q,c)&=\prod_{\underset{a'b'c'=\ell}{x+y+z=\alpha}}(q,q^xa')(q,q^yb')(q,q^zc')\\
        &=\prod_{\underset{a'b'c'=\ell}{x+y+z=\alpha}}(q,q^x)(q,q^y)(q,q^z)\\
        &=\left[\prod_{x+y+z=\alpha}(q,q^x)(q,q^y)(q,q^z)\right]^{d_3(\ell)}\\
        &=q^{\frac{3\alpha(\alpha+1)d_3(\ell)}{2}}.
\end{align*} 
Therefore, 
\begin{align*}
        A_{h/q}(n)&=q^2\log_q\left[\prod_{abc=n} (q,a)(q,b)(q,c)\right]-(q^2+q)d_3(n)\\
        &=\frac{3q^2\alpha(\alpha+1)d_3(\ell)}{2}-(q^2+q)d_3(n).
\end{align*} Since $d_3$ is multiplicative, $d_3(n)=d_3(\ell)d_3(q^\alpha)=\frac{(\alpha+2)(\alpha+1)}{2}d_3(\ell)$. 
\begin{align*}
        A_{h/q}(n)&=\frac{3q^2\alpha d_3(n)}{\alpha+2}-(q^2+q)d_3(n)\\
        &=d_3(n)\left[2q^2\left(1-\frac{3}{\alpha+2}\right)-q\right]\\
        &=d_3(n)\cdot \begin{cases}
          -q&\text{if }\alpha=1,\\
          C_{\alpha}q^2 -q&\text{if }\alpha\geq 2,
        \end{cases}
\end{align*} 
where $C_{\alpha}:=2-\frac{6}{\alpha+2}\neq 0$. Notice that $A_{h/q}(n)$ does not depend on $h$. Since $d_3(n)\ll_{\varepsilon} n^{\varepsilon}$ for any $\varepsilon>0$, $A_{h/q}(n)\ll_{\varepsilon}n^{\varepsilon}(q^2,n).$ For the second bound, if we use $(q,m)=1$,
\begin{align*}
        \overline{A_{h/q}(m)}&=\frac{1}{2}\sum_{abc=m}\overline{R^+_{a,b,c}(h/q)}\\
        &=\frac{q}{2}\sum_{abc=m}\sum_{\underset{xy\equiv c\bar h\pmod{q}}{1\leq x,y\leq q}}e\left(\frac{-ax-by}{q}\right)+\frac{q}{2}\sum_{abc=m}\sum_{\underset{xy\equiv c\bar h\pmod{q}}{1\leq x,y\leq q}}e\left(\frac{-ax-by}{q}\right)\\
        &=\frac{q}{2}\sum_{abc=m}\sum_{x=1}^{q-1}e\left(\frac{-ax-bc\overline{hx}}{q}\right)+\frac{q}{2}\sum_{abc=m}\sum_{x=1}^{q-1}e\left(\frac{-ax+bc\overline{hx}}{q}\right)\\
        &=\frac{q}{2}\sum_{abc=m}\sum_{u=1}^{q-1}e\left(\frac{u+m\overline{hu}}{q}\right)+\frac{q}{2}\sum_{abc=m}\sum_{u=1}^{q-1}e\left(\frac{u-m\overline{hu}}{q}\right)\\
        &=\frac{qd_3(m)\overline{K(1,m\bar h;q)}}{2}+\frac{qd_3(m)\overline{K(1,-m\bar h;q)}}{2},
\end{align*} 
where $K(n,m;q)=\sideset{}{'}\sum\limits_{r=1}^qe\left(\frac{nr+m\bar r}{q}\right)$ is Kloosterman’s sum. Therefore, 
\begin{equation*}
        \sum_{h=1}^{q-1}A_{h/q}(n)\,\overline{A_{h/q}(m)}=\frac{qd_3(m)A_{1/q}(n)}{2}\sum_{h=1}^{q-1}\left(\overline{K(1,m\bar h;q)}+\overline{K(1,-m\bar h;q)}\right)
\end{equation*} 
because $A_{h/q}(n)$ does not depend on $h$. Now,
\begin{align*}
    \sum_{h=1}^{q-1}\overline{K(1,m\bar h;q)}=\sum_{h=1}^{q-1}\sum_{u=1}^{q-1}e\left(\frac{-u-m\overline{hu}}{q}\right)&=\sum_{u=1}^{q-1}e\left(\frac{-u}{q}\right)\sum_{h=1}^{q-1}e\left(\frac{-m\overline{uh}}{q}\right)\\
    &=-\sum_{u=1}^{q-1}e\left(\frac{-u}{q}\right)=1.
\end{align*}
Therefore, we have
\begin{equation*}
        \sum_{h=1}^{q-1}A_{h/q}(n)\overline{A_{h/q}(m)}=qd_3(m)A_{1/q}(n)\ll_{\varepsilon} (mn)^{\varepsilon}q(q^2,n).
\end{equation*}
\end{proof}
The next lemma is essential for rewriting sums involving $A_{h/q}$ in terms of $d_3(m)d_3(n)c_q(m\pm n)$.
\begin{lemma}\label{lemma3}
    Suppose $q$ is prime and let $A_{h/q}$ be as in Lemma \ref{lemma1}. If $(nm,q)=1$, then $$\sum\limits_{h=1}^{q-1}A_{h/q}(n)\overline{A_{h/q}(m)}=\frac{q^3d_3(m)d_3(n)\left[c_q(m-n)+c_q(m+n)\right]}{2}-q^2d_3(m)d_3(n).$$
    The assertion remains valid when $A_{h/q}$ is substituted with $B_{h/q}$.
\end{lemma}
\begin{proof}
    Since $n$ and $m$ are coprime to $q$, as we found in the proof of Lemma \ref{lemma2},
    \begin{align*}
        A_{h/q}(n)=\frac{qd_3(n)(K(1,n\bar h;q)+K(1,-n\bar h;q))}{2}.
    \end{align*}
     Therefore,
\begin{align*}
    \sum\limits_{h=1}^{q-1}A_{h/q}(n)\overline{A_{h/q}(m)}&=\frac{q^2d_3(n)d_3(m)}{4}\\
    &\cdot \sum_{h=1}^{q-1}(K(1,n\bar h;q)+K(1,-n\bar h;q))(\overline{K(1,m\bar h;q)}+\overline{K(1,-m\bar h;q)}).
\end{align*} 
     If we calculate the following sum,
    \begin{align*}
        \sum_{h=1}^{q-1}K(1,n\bar h;q)\overline{K(1,m\bar h;q)}&=\sum_{h=1}^{q-1}\sum_{u=1}^{q-1}\sum_{v=1}^{q-1}e\left(\frac{u+m\bar h\bar  u-v-n\bar h\bar v}{q}\right)\\
        &=\sum_{u=1}^{q-1}\sum_{v=1}^{q-1}e\left(\frac{u-v}{q}\right)\sum_{h=1}^{q-1}e\left(\frac{(m\bar u-n\bar v)\bar h}{q}\right)\\
        &=\sum_{u=1}^{q-1}\sum_{v=1}^{q-1}e\left(\frac{u-v}{q}\right)\sum_{h=1}^{q-1}e\left(\frac{(m\bar u-n\bar v) h}{q}\right)\\
        &=q\underset{mv\equiv nu\pmod{q}}{\sum_{u=1}^{q-1}\sum_{v=1}^{q-1}}e\left(\frac{u-v}{q}\right)-\underset{=(-1)\cdot (-1)=1}{\underbrace{\sum_{u=1}^{q-1}\sum_{v=1}^{q-1}e\left(\frac{u-v}{q}\right)}}\\
        &=-1+q\underset{mv\equiv nu\pmod{q}}{\sum_{u=1}^{q-1}\sum_{v=1}^{q-1}}e\left(\frac{u-v}{q}\right)\\
        &=-1+q\sum_{u=1}^{q-1}e\left(\frac{u(1-\bar m n) }{q}\right)\\
        &=-1+qc_q(1-\bar{m}n)\\
        &=-1+qc_q(m-n).
    \end{align*} $c_q(\cdot)$ is real-valued since $q$ is a prime. Therefore, $$c_q(n-m)=c_q(m-n),\qquad \text{and}\qquad c_q(-m-n)=c_q(m+n).$$ Hence, we can write this result as
    \begin{align*}
        \sum\limits_{h=1}^{q-1}A_{h/q}(n)\overline{A_{h/q}(m)}&=\frac{q^2d_3(n)d_3(m)}{4}\left[-4+2qc_q(m-n)+2qc_q(m+n)\right]
    \end{align*}
    $$\sum\limits_{h=1}^{q-1}A_{h/q}(n)\overline{A_{h/q}(m)}=\frac{q^3d_3(m)d_3(n)\left[c_q(m-n)+c_q(m+n)\right]}{2}-q^2d_3(m)d_3(n).$$
\end{proof}
\begin{lemma}\label{lemma4}
Let $U(x)$ be as in Lemma \ref{lemma1}. Then for any fixed integer $K\geq 1$ and $x >0$
\begin{equation*}
    U(x)\ =\ \sum_{j=1}^{K}\frac{c_j \cos(6x^{1/3})+d_j \sin{(6x^{1/3})}}{x^{j/3}}\ +\ O\Big(\frac{1}{x^{(K+1)/3}}\Big)
\end{equation*}
with absolute constants $c_j,d_j$.
\end{lemma}
\begin{proof}
    This is Lemma 3 of Ivić's \cite{Iv} paper.
\end{proof}
\begin{lemma}\label{lemma5}
    Take $1\leq Y\leq x$ and let $w:[0,\infty)\to \mathbb{R}$ be a smooth function satisfying
    \begin{align*}
        w(t) = 
        \begin{cases} 
            0, & t \in [0, x-Y]\cup [2x+Y,\infty), \\ 
            1, & t \in [x, 2x]. 
        \end{cases} \hspace{2cm} w^{(j)}(t) \ll \frac{1}{Y^j} \quad (j \geq 0).
\end{align*}
If $Nx\gg 1$ then for any $j\in \mathbb{N}$
\begin{equation*}
    \hat{w}_q(n) \ \ll\ \frac{Y}{(Nx)^{1/3}} \Big( \frac{x^2}{NY^3}\Big)^{j/3},
\end{equation*}
where $N$ and $\hat{\omega}_q(n)$ are as in Lemma \ref{lemma1}. An analogous bound remains valid when $\hat{w}_q$ is replaced by $\hat{\nu}_q$.
\end{lemma}
\begin{proof}
Let $f(t) = 6t^{1/3}$ and let
\begin{equation*}
    M(X)\ :=\ \frac{e(f(X))}{X^{1/3}} \hspace{1cm} \text{and} \hspace{1cm} \mathcal{M}(N)\ :=\ \int_{0}^{\infty} w(t)\, M(Nt)\, dt.
\end{equation*}
Substituting the definition of $M$ into the integral, we obtain
\begin{equation*}
    N^{1/3} \mathcal{M}(N) = \int_{0}^{\infty} w(t)\, \frac{e(f(Nt))}{t^{1/3}}\, dt.
\end{equation*}
We now differentiate $e(f(Nt))$, which is
\begin{equation*}
     \frac{d(e(f(Nt)))}{dt}\ =\ 2\pi iN \underbrace{f'(Nt)}_{=2(Nt)^{-2/3}} e(f(Nt)).
\end{equation*}
This implies that
\begin{equation*}
    e(f(Nt))\ =\ \frac{1}{4\pi i N^{1/3}}\, t^{2/3}\, \frac{d}{dt}e(f(Nt)).
\end{equation*}
Substituting this expression back into the integral and integrating by parts, we obtain
\begin{equation*}
    N^{1/3} \mathcal{M}(N) = \frac{1}{4\pi i N^{1/3}} \int_{0}^{\infty} \frac{d}{dt}(w(t)\, t^{1/3})\, e(f(Nt))\, dt.
\end{equation*}
So,
\begin{align*}
    \leq \frac{1}{4N^{1/3}} \Big| \underbrace{\int_{0}^{\infty} \frac{w'(t)\, e(f(Nt))}{t^{1/3}}\, t^{2/3} dt}_{(1)} + \underbrace{\int_{0}^{\infty} \frac{w(t)\, e(f(Nt))}{t^{1/3}}\, \frac{dt}{t^{1/3}}}_{(2)} \Big|.
\end{align*}
Repeating the same argument $j$ times for the integrals (1) and (2), we obtain
\begin{equation}\label{boundlemma5}
    N^{1/3} |\mathcal{M}(N)| \leq \frac{1}{4^j N^{j/3}}\, \underset{2^j\ \text{terms}}{\underbrace{\sum_{A+B=j}}} \Big| \int_{0}^{\infty} \frac{w^{(A)}(t)\, e(f(Nt))}{t^{1/3}}\, \frac{t^{2A/3}}{t^{B/3}}\, dt \Big|.
\end{equation}
We now bound the right-hand side in terms of $x$, $Y$, and $N$. From the lemma setup, we have $w^{(A)} \ll 1/Y^A$. The support of $w^{(A)}$ is contained in $t \in [x - Y, 2x + Y]$. The support length is $\ll Y$. Fixing $A, B$ with $A + B = j$, we estimate
\begin{equation*}
    \int_{0}^{\infty} \frac{w^{(A)}(t)\, e(f(Nt))}{t^{1/3}}\, \frac{t^{2A/3}}{t^{B/3}}\, dt \ll \frac{1}{Y^A}\, x^{(2A - B - 1)/3} Y = x^{(2A - B - 1)/3}\, Y^{1 - A}.
\end{equation*}
Plugging this back into the whole sum, we have
\begin{equation*}
    N^{1/3} |\mathcal{M}(N)| \ll \frac{1}{N^{j/3}} \sum_{A + B = j} x^{(2A - B - 1)/3}\, Y^{1 - A}.
\end{equation*}
Recall $B = j - A$, so the term becomes
\begin{equation*}
    \ll \frac{1}{N^{j/3}}\, x^{(3A - j - 1)/3}\, Y^{1 - A}.
\end{equation*}
The largest contribution comes from the term $A = j$. We obtain
\begin{equation*}
    N^{1/3} |\mathcal{M}(N)| \ll \frac{Y}{x^{1/3}} \left( \frac{x^2}{N Y^3} \right)^{j/3}.
\end{equation*}
This gives the desired bound for $\mathcal{M}(N)$. From Lemma \ref{lemma4} we have 
\begin{equation*}
    \left|M(X)\  -\  U(X)\right|\ \ll  \frac{1}{X^{(K+1)/3}},
\end{equation*}
for $K\geq 1$. Hence,  we get
\begin{equation*}
    \hat{w}_q(n)\ =\ \int_{0}^{\infty} w(t)\  U(Nt)\ dt\ \ll \  \frac{Y}{(Nx)^{1/3}} \left( \frac{x^2}{N Y^3} \right)^{j/3}.
\end{equation*}
\end{proof}
Now, we collect the lemmas about Dirichlet $L$-functions that we will use later in the proof of the main theorem.

\begin{lemma}\label{bound of L^3}
Let $\chi$ be a primitive Dirichlet character modulo $q$. For $y > 1$ and $T \geq 2$, the second moment of $L$-functions on the critical line satisfies the following bound:
\begin{equation*}
    \int_{y}^{T} \frac{|L(1/2 + it, \chi)|^2}{t} \, dt \ll\left(1 + \frac{q^{1/3}}{y^{2/3}} + \frac{q^{1/2}}{y}\right)\log^5(qT).
\end{equation*}
\end{lemma}
\begin{proof}
The proof relies on a well-known second moment bound for Dirichlet $L$-functions due to Motohashi \cite{Mo}, namely,
  \begin{equation*}
           \int_{0}^{T} |L(1/2 + it, \chi)|^2 \, dt \ll T+( q^{1/3} T^{1/3}+ q^{1/2}) \log^4 (qT).
  \end{equation*}
To estimate the integral 
\begin{equation*}
    \int_{y}^{T} \frac{|L(1/2 + it, \chi)|^2}{t}  dt,
\end{equation*}
we proceed by a dyadic decomposition
\begin{align*}
\int_{y}^{T} \frac{|L(1/2 + it, \chi)|^2}{t}  dt &\ll \log T \max_{y \leq U \leq T} \left\{ \int_{U}^{2U} \frac{|L(1/2 + it, \chi)|^2}{t} \, dt \right\} \\
&\ll \log T \max_{y \leq U \leq T} \left\{ \frac{1}{U} \int_{U}^{2U} |L(1/2 + it, \chi)|^2 \, dt \right\} \\
&\ll \log T \max_{y \leq U \leq T} \left\{ \frac{1}{U} \left( U + q^{1/3} U^{1/3} + q^{1/2} \right)\log^4(qU) \right\} \\
&\ll \left(1 + \frac{q^{1/3}}{y^{2/3}} + \frac{q^{1/2}}{y}\right)\log^5 (qT),
\end{align*}
which establishes the desired bound.
\end{proof}
\begin{lemma}\label{xxxxxx}
For any $\varepsilon>0$, $$\sum_{\substack{\chi(q)\\ \chi\neq \chi_0}}\left(\int_{1}^T\frac{|L(1/2+it,\chi)|^3}{t}dt\right)^2\ll_{\varepsilon} q^{11/8}(qT)^{\varepsilon}.$$
\end{lemma}
\begin{proof}
    By Cauchy-Schwarz inequality,
    \begin{multline*}
        \sum_{\substack{\chi(q)\\ \chi\neq \chi_0}}\left(\int_{1}^T\frac{|L(1/2+it,\chi)|^3}{t}\,dt\right)^2\\ \leq \sum_{\substack{\chi(q)\\ \chi\neq \chi_0}}\left(\int_{1}^T\frac{|L(1/2+it,\chi)|^2}{t}\,dt\right)\left(\int_{1}^T\frac{|L(1/2+it,\chi)|^4}{t}\,dt\right).
    \end{multline*}
We split the integral at a parameter $y\ge 1$. For the range $1\le t\le y$, we use the subconvexity estimate of Petrow and Young \cite[Corollary 1.3]{PY}. Since $q$ is prime, it is in particular cube-free, and hence their result applies. More precisely, for any $\varepsilon>0$,
\begin{equation}
    L(1/2+it,\chi)\ll_\varepsilon (q(1+|t|))^{1/6+\varepsilon}.
\end{equation}\label{petrow young boundu}
Therefore,
\begin{equation*}
    |L(1/2+it,\chi)|^2\ll_\varepsilon q^{1/3+\varepsilon}(1+t)^{1/3+2\varepsilon},
\end{equation*}
and hence
\begin{equation*}
    \int_1^y \frac{|L(1/2+it,\chi)|^2}{t}\,dt
    \ll_\varepsilon
    q^{1/3+\varepsilon}\int_1^y t^{-2/3+2\varepsilon}\,dt
    \ll_\varepsilon
    q^{1/3}y^{1/3}.
\end{equation*}
For the remaining range $y\le t\le T$, we use Lemma \ref{bound of L^3}, which gives
\begin{equation*}
    \int_y^T \frac{|L(1/2+it,\chi)|^2}{t}\,dt
    \ll_\varepsilon
    \bigl(1+q^{1/3}y^{-2/3}+q^{1/2}y^{-1}\bigr)(qT)^\varepsilon.
\end{equation*}
Therefore,
\begin{equation*}
     \int_{1}^T\frac{|L(1/2+it,\chi)|^2}{t}dt\, \ll_{\varepsilon} \left(q^{1/3}y^{1/3}+1+q^{1/3}y^{-2/3}+q^{1/2}y^{-1}\right)(qT)^{\varepsilon}.
\end{equation*}
    If we choose $y=q^{1/8}$, then $$\int_{1}^T\frac{|L(1/2+it,\chi)|^2}{t}\,dt\ll_{\varepsilon}(qT)^{\varepsilon}q^{3/8}.$$ Combining this with the following bound from Montgomery \cite[Theorem 10.1]{Mon}
    \begin{equation*}
         \sum_{\substack{\chi(q)\\ \chi\neq \chi_0}}\int_{1}^T\frac{|L(1/2+it,\chi)|^4}{t}\,dt\ \ll \ q
    \end{equation*}
    we obtain
    \begin{align*}
        \sum_{\substack{\chi(q)\\ \chi\neq \chi_0}}\left(\int_{1}^T\frac{|L(1/2+it,\chi)|^2}{t}\,dt\right)&\left(\int_{1}^T\frac{|L(1/2+it,\chi)|^4}{t}\,dt\right)\\
        &\ll_{\varepsilon} (qT)^{\varepsilon}q^{3/8}\sum_{\substack{\chi(q)\\ \chi\neq \chi_0}}\int_{1}^T\frac{|L(1/2+it,\chi)|^4}{t}\,dt\\
        &\ll_{\varepsilon}(qT)^{\varepsilon}q^{11/8}.
    \end{align*}
\end{proof}
\begin{lemma}\label{mu q c_q}
    Let $m \in \mathbb{Z}^+$ with prime $q$. Then the following sum satisfies
    \begin{equation*}
        \sideset{}{'}\sum_{r=1}^{q}  c_q(r-m) = \mu(q)c_q(m).
    \end{equation*}
\end{lemma}
\begin{proof}
    Using the definition of Ramanujan's sum, we can expand the sum as
    \begin{equation*}
        \sideset{}{'}\sum_{r=1}^q \sideset{}{'}\sum_{k=1}^q e\left(\frac{(r-m)k}{q}\right)= \underbrace{\sideset{}{'}\sum_{r=1}^q e\left(\frac{rk}{q}\right)}_{(1)}\  \underbrace{\sideset{}{'}\sum_{k=1}^q e\left(\frac{-mk}{q} \right)}_{(2)}.
    \end{equation*}
    By the definition $(1)=c_q(k)$. Since $(r,q)=1$ and together with $(k,q)=1$ imply that $(1)=\mu(q)$. Similarly, $(2)=c_q(-m)$. Since $c_q(.)$ is an even function, $(2)=c_q(m)$ follows. This completes the proof.    
\end{proof}
\begin{lemma}\label{gauss sum part}
    Let $q$ be a positive integer and $\chi$ be a Dirichlet character modulo $q$. For any integer $m$, the following identity holds
    \begin{equation*}
        \sideset{}{'}\sum_{r=1}^q \overline{\chi}(r)c_q(r-m) = q \chi(-m).
    \end{equation*}
\end{lemma}
\begin{proof}
    Let $\chi$ be a Dirichlet character modulo $q$. By the definition, we have
    \begin{equation*}
        \sideset{}{'}\sum_{r=1}^q \overline{\chi}(r)c_q(r-m) = \sideset{}{'}\sum_{r=1}^q \overline{\chi}(r) \sideset{}{'}\sum_{x=1}^q e\left(\frac{x(r-m)}{q} \right).
    \end{equation*}
    Interchanging the order of summation, we obtain
    \begin{equation*}
        \sideset{}{'}\sum_{x=1}^q e\left(\frac{-xm}{q} \right) \sideset{}{'}\sum_{r=1}^q \overline{\chi}(r) e\left(\frac{xr}{q} \right).
    \end{equation*}
    Recall the definition of the Gauss sum $\tau(\chi)=\sum_{a=1}^q \chi(a)e(a/q)$. For a primitive character, we have the identity
    \begin{equation*}
        \sum_{r=1}^q \overline{\chi}(r) e\left(\frac{rx}{q}\right) = \chi(x)\tau(\overline{\chi}).
    \end{equation*}
    Substituting this into the previous expression, we obtain
    \begin{equation*}
         \tau(\overline{\chi})\underbrace{\sideset{}{'}\sum_{x=1}^q \chi(x)e\left(\frac{-xm}{q} \right)}_{=\overline{\chi}(-m)\tau(\chi)}. 
    \end{equation*}
    Thus,
    \begin{equation*}
        \tau(\overline{\chi})\tau(\chi)\overline{\chi}(-m)=q\chi(-1)\overline{\chi}(-m)=q\chi(-m),
    \end{equation*}
    as desired.
\end{proof}
\begin{lemma}\label{lemma6}
For $N,M \geq q$ with $q$ prime, for any $\varepsilon>0$,
\begin{equation*}
    \sum_{\substack{n\leq N\\ m\leq M\\ (nm,q)=1}} d_3(n) d_3(m) c_q(n-m) \ll_{\varepsilon}\left( \frac{NM}{q}+ \sqrt{NM}\cdot q\left(q^{3/8}+q^{-7/16}M^{-1/2}\right)\right)(NM)^\varepsilon.
\end{equation*}
\end{lemma}
\begin{proof}
We begin by analyzing the sum $\sum\limits_{n\leq N} d_3(n) c_q(n-m)$. Assume that $N\leq M$ and  $T\geq q^{5/8} M$. The trivial bound for this sum is $\ll_{\varepsilon} N^{1+\varepsilon}$. To refine this, we proceed as follows
    \begin{equation*}
        \sum_{\substack{n\leq N\\ (n,q)=1}} d_3(n) c_q(n-m)\ =\ \sideset{}{'}\sum_{r=1}^q c_q(r-m) \sum_{\substack{n\leq N\\ n\equiv r(q)}} d_3(n).
    \end{equation*}
    By the orthogonality of the Dirichlet characters, this becomes
    \begin{equation*}
        \frac{1}{\phi(q)} \sum_{\chi  (q)}\  \sideset{}{'}\sum_{r=1}^q \bar{\chi}(r)\ c_q(r-m)\ \sum_{n\leq N} d_3(n) \chi(n).
    \end{equation*}    
    The sum splits into two parts
    \begin{equation}\label{main lemma first step}
         \underbrace{\frac{1}{\phi(q)} \  \sideset{}{'}\sum_{r=1}^q \ c_q(r-m) \ \sum_{n\leq N} d_3(n)}_{ (1)} \ + \underbrace{\frac{1}{\phi(q)} \sum_{\chi\neq \chi_0}\  \sideset{}{'}\sum_{r=1}^q \bar{\chi}(r)\ c_q(r-m)\ \sum_{n\leq N} d_3(n) \chi(n)}_{(2)}.   
    \end{equation}
The first term (1) is the contribution of the principal character. Using Lemma \ref{mu q c_q} we have
\begin{equation}\label{main lemma eqn (1)}
    (1) = \frac{\mu(q)c_q(m)}{\phi(q)} \sum_{n\leq N} d_3(n) \ll_{\varepsilon} \frac{N^{1+\varepsilon}}{q}.
\end{equation}
For the second term $(2)$, Lemma \ref{gauss sum part} gives
\begin{equation}\label{main lemma eqn (2)}
    (2)= \frac{q}{\phi(q)}\sum_{\chi \neq \chi_0} \chi(-m) \underbrace{\sum_{n\leq N} d_3(n) \chi(n)}_{:=f(N)}.
\end{equation}
Substituting (\ref{main lemma eqn (1)}) and (\ref{main lemma eqn (2)}) into (\ref{main lemma first step}), and then multiplying by $\sum_{m\leq M} d_3(m)$, we obtain
\begin{equation}\label{goal of the main lemma}
    \sum_{\substack{n\leq N\\m\leq M \\ (nm,q)=1}} d_3(n)d_3(m)c_q(n-m) \ll_{\varepsilon} \frac{N^{1+\varepsilon}}{q}\sum_{m\leq M}d_3(m) + \frac{q}{\phi(q)}\sum_{\chi \neq \chi_0} f(N)f(M).
\end{equation}
To bound the sum $\sum_{\chi\neq \chi_0} f(N)f(M)$ we use Perron's formula. For $f(N)$, we obtain
\begin{align*}
f(N) = \sum_{n \leq N} d_{3}(n)   \chi(n) = \int\limits_{c\pm iT} \frac{(L(s, \chi))^3\ N^s}{s} \, ds + O\left( \frac{N}{T} \right).
\end{align*}
where  $c = 1 + \frac{1}{\log N}$  and $T\to \infty$. Shifting the contour to Re$(s)=1/2$, we use Cauchy's theorem to express the integral as
\begin{equation*}
    \int_{c \pm iT} = \int_{1/2-iT}^{1/2+iT} + \int_{1/2+iT}^{c+iT} + \int_{c-iT}^{1/2-iT}. 
\end{equation*}
The last two integrals are horizontal contributions. The main contribution comes from the integral along Re$(s)=1/2$, which is
\begin{equation*}
    \int_{1/2-iT}^{1/2+iT}L(s,\chi)^3\frac{N^s}{s}\, ds\ll N^{1/2}\int_{-T}^T|L(1/2+it,\chi)|^3\frac{dt}{\max(1,|t|)}.
\end{equation*}
We can consider it as
    \begin{align}\nonumber
        \int_{-T}^T|L(1/2&+it,\chi)|^3\frac{dt}{\max(1,|t|)}\\ \nonumber
        &=\int_{0}^T|L(1/2+it,\chi)|^3\frac{dt}{\max(1,|t|)}+\int_{-T}^0|L(1/2+it,\chi)|^3\frac{dt}{\max(1,|t|)}\\ \nonumber
        &=\int_{1}^T|L(1/2+it,\chi)|^3\frac{dt}{t}+\int_{-T}^{-1}|L(1/2+it,\chi)|^3\frac{dt}{-t}+O(1)\\ 
        &=\int_{1}^T|L(1/2+it,\chi)|^3\frac{dt}{t}+\int_{1}^{T}|L(1/2-it,\chi)|^3\frac{dt}{t}+O(1).
    \end{align} 
We next bound the horizontal integral as follows    
\begin{equation}
    \int_{1/2+iT}^{c+iT} L(s,\chi)^3 \frac{N^s}{s}ds \ll_{\varepsilon} \frac{1}{T}\int_{1/2}^1 |L(\sigma+iT,\chi)|^3 N^{\sigma} d\sigma \ll_{\varepsilon}  \left(\frac{q}{T}\right)^{1/2} (N+N^{1/2}),
\end{equation}
where we used the Petrow–Young subconvexity bound 
\begin{equation}\label{2 petrow young boundu}
    L(1/2+it,\chi)\ll_\varepsilon (q(1+|t|))^{1/6+\varepsilon}
\end{equation}
from (\ref{petrow young boundu}). This can be done for the other horizontal integral in a similar way. 
We now estimate $\sum_{\chi\ \neq \chi_0} f(N)f(M)$. We have
\begin{multline}\label{mainpart}
     \ll_{\varepsilon} \sum_{\chi\neq \chi_0} \left(\sqrt{N} \int_{1}^T \frac{\left(L(1/2+it,\chi)\right)^3}{t}\  dt \ +\ O\left(\underbrace{\left(\frac{q}{T}\right)^{1/2} (N+N^{1/2})}_{\text{Horizontal contribution}}+\frac{N}{T}\right)\right) \\
    \cdot\ \left(\sqrt{M} \int_{1}^T \frac{\left(L(1/2+it,\chi)\right)^3}{t}\  dt \ +\ O\left(\underbrace{\left(\frac{q}{T}\right)^{1/2} (M+M^{1/2})}_{\text{Horizontal contribution}}+\frac{M}{T}\right)\right).
\end{multline}
We now estimate the resulting terms one by one. 
\begin{equation*}
    \sum_{\chi\neq\chi_0} f(N)f(M).
\end{equation*}
By Lemma \ref{xxxxxx},
\begin{multline}\label{peron part 1}
    \sum_{\chi \neq \chi_0} \left(\sqrt{N} \int_{1}^T \frac{\left(L(1/2+it,\chi)\right)^3}{t}\  dt\right) \left(\sqrt{M} \int_{1}^T \frac{\left(L(1/2+it,\chi)\right)^3}{t}\  dt\right)\\ 
   \leq \sqrt{NM}\sum
   _{\chi \neq \chi_0} \left(\int_{1}^T\frac{|L(1/2+it,\chi)|^3}{t}\,dt\right)^2\ \ll_{\varepsilon} \sqrt{NM}\cdot q^{11/8}(qT)^{\varepsilon}.
\end{multline}
Using Cauchy-Schwarz and Lemma \ref{xxxxxx}
\begin{multline}\label{peron part 2}
   O\left( \sum_{\chi \neq \chi_0} \left(\sqrt{N} \int_{1}^T \frac{\left(L(1/2+it,\chi)\right)^3}{t}\  dt\right) \left( \left(\frac{q}{T}\right)^{1/2} (M+M^{1/2})+\frac{M}{T}\right)\right)\\
  \le O\left(\sqrt{N} \left(\sum_{\chi \neq \chi_0} \left( \int_{1}^T \frac{\left(L(1/2+it,\chi)\right)^3}{t}\  dt\right)^2\right)^{1/2}\left(\sum_{\chi\neq\chi_0}\left( \left(\frac{q}{T}\right)^{1/2} (M+M^{1/2})+\frac{M}{T}\right)^2 \right)^{1/2}\right)\\
  \ll_{\varepsilon} \sqrt{N} (q^{11/8})^{1/2} \left(\frac{q^2}{T} (M^2+M)+\frac{qM^2}{T^2}\right)^{1/2}\\ \ll_{\varepsilon} \sqrt{NM} (\ q^{11/8}\ )^{1/2}\cdot\ \frac{q}{T^{1/2}}\sqrt{M} + \sqrt{N}(\ q^{11/8}\ )^{1/2} \cdot\ \frac{q^{1/2}M}{T}.
\end{multline}
Similarly, 
\begin{multline}\label{peron part 3}
      O\left( \sum_{\chi \neq \chi_0} \left(\sqrt{M} \int_{1}^T \frac{\left(L(1/2+it,\chi)\right)^3}{t}\  dt\right) \left( \left(\frac{q}{T}\right)^{1/2} (N+N^{1/2})+\frac{N}{T}\right)\right)\\
      \ll_{\varepsilon} \sqrt{NM} (\ q^{11/8}\ )^{1/2}\cdot\ \frac{q}{T^{1/2}}\sqrt{N} + \sqrt{M}(\ q^{11/8}\ )^{1/2} \cdot\ \frac{q^{1/2}N}{T}.
\end{multline}
Lastly,
\begin{multline}\label{peron part 4}
    O\left( \sum_{\chi\neq \chi_0} \left(\left(\frac{q}{T}\right)^{1/2} (N+N^{1/2})+\frac{N}{T}\right)\left(\left(\frac{q}{T}\right)^{1/2} (M+M^{1/2})+\frac{M}{T}\right)\right)\\
    \ll_{\varepsilon} MN\frac{q^2}{T}+MN\frac{q^{2/3}}{T^{2/3}}+MN\frac{q}{T^2}.
\end{multline}
Therefore, combining all the in (\ref{peron part 1}), (\ref{peron part 2}), (\ref{peron part 3}) and (\ref{peron part 4}), we get a bound for the (\ref{mainpart})  with choosing $T=q^{5/8} M$ as
\begin{align}\label{f(N,M) bound}
    \sum_{\chi\neq \chi_0} f(N)f(M)\ \ll q\sqrt{NM} \left(q^{3/8}+q^{-7/16}M^{-1/2}\right).
\end{align}
Substituting (\ref{f(N,M) bound}) into (\ref{goal of the main lemma}), we obtain
\begin{equation*}
    \sum_{\substack{n,m\leq N, M\\(nm,q)=1}} d_3(n) d_3(m) c_q(n-m) \ \ll \ (NM)^\varepsilon \Big( \frac{NM}{q}+ q\sqrt{NM}\ (\ q^{3/8}\ +\ q^{-7/16}M^{-1/2}\ ) \Big).
\end{equation*}
\end{proof}
\begin{lemma}\label{lemma9}
    Take $q,Q,T\leq x^{O(1)}$ with $q$ prime, take $P$ any polynomial, and write $\sum_{n,m}^{\sharp}$ for a sum subject to $n,m>q$ and $(nm,q)=1$. Then up to an error $\ll x^{\varepsilon}$
\begin{multline*}
    \sum_{n,m \leq Q}^{\sharp} d_{3}(n)d_{3}(m)c_{q}(n-m)P(\log n/Q)P(\log m/Q) \ll\ \frac{Q^2}{q}+q^{11/8}Q+q^{9/16}\sqrt{Q},
    \end{multline*}
    \begin{multline*}
    \sum_{\substack{n\leq Q\\Q < m \leq T}}^{\sharp} \frac{d_{3}(n)d_{3}(m)c_{q}(n-m)P(\log n/Q)}{m^{2/3}} e\left((n/Q)^{1/3}\right)\\ \ll\  \frac{(Tx)^{1/3}}{q}\Big( \frac{QT^{1/3}}{q}+q^{11/8}Q^{1/3}+\frac{q^{9/16}}{Q^{1/6}}\Big),    
    \end{multline*}
    \begin{multline*}
    \sum_{Q < n,m \leq T}^{\sharp} \frac{d_{3}(n)d_{3}(m)c_{q}(n-m)}{(nm)^{2/3}} e\left((n/Q)^{1/3} - (m/Q)^{1/3}\right)\\ \ll \frac{(Tx)^{2/3}}{q^2} \Big( \frac{T^{2/3}}{q}+\frac{q^{11/8}}{Q^{1/3}}+\frac{q^{9/16}}{Q^{5/6}}\Big).
\end{multline*}
\end{lemma}
\begin{proof}
    We omit the factors $x^{\varepsilon}$ from the notation and write $K=1/3$ and $d=2/3$. We have

\begin{center}
    \begin{tabular}{c c c c}
       $f(t) \ll 1$ &  $f^{\prime}(t) \ll \dfrac{1}{t}$  &\hspace{0.25 cm} $\text{for } f(t) = P(\log t/Q)$ \\
       $f(t) \ll \dfrac{1}{t^{d}}$ &\hspace{0.9 cm} $f^{\prime}(t) \ll \dfrac{(t/Q)^{K} + 1}{t^{d+1}}$  & $\hspace{0.3 cm}\text{for } f(t) = \dfrac{e\left((t/Q)^{1/3}\right)}{t^{d}}$
    \end{tabular}
\end{center}
From Lemma \ref{lemma6}, we have
\begin{align}\label{lemma7.1}
    &\sum_{\begin{subarray}{c}n \sim N \\ m \sim M\end{subarray}}^{\sharp} d_{3}(n)d_{3}(m)c_{q}(n-m)f(n)g(m)\\
    &\ll \max_{\substack{t \sim N \\t^{\prime} \sim M}} \left(|f(t)g(t^{\prime})| + N|f^{\prime}(t)g(t)| + M|f(t)g^{\prime}(t^{\prime})| + NM|f^{\prime}(t)g^{\prime}(t^{\prime})|\right)\\ \nonumber
    &\hspace{2cm}\cdot\ \Big( \frac{NM}{q}+ q\sqrt{NM}\ (\ q^{3/8}\ +\ q^{-7/16}M^{-1/2}\ ).\nonumber
\end{align}
We now consider (\ref{lemma7.1}) case by case, according to the choices of $f$ and $g$.
\begin{enumerate}
    \item For $f,g=P(\log t/Q)$:\\
    \begin{align*}
        \ll\  \frac{NM}{q}+ q\sqrt{NM}\ (\ q^{3/8}\ +\ q^{-7/16}M^{-1/2}\ ).
    \end{align*}
    \item For $f=P(\log t/Q)$ and $g=\dfrac{e((t/Q)^{K})}{t^d}$:\\
    \begin{equation*}
        \ll\  \left((M/Q)^{1/3} + 1\right) \left(\frac{NM^{1-d}}{q} + q\sqrt{N}M^{1/2-d}(\ q^{3/8}\ +\ q^{-7/16}M^{-1/2}\ )\right).
    \end{equation*}
    \item For $f,g=\dfrac{e((t/Q)^{K})}{t^d}$:\\
    \begin{multline*}
     \ll\ \left((N/Q)^{1/3} + 1\right) \left((M/Q)^{1/3} + 1\right)\\ \left(\frac{(NM)^{1-d}}{q} + q(NM)^{1/2-d} (\ q^{3/8}\ +\ q^{-7/16}M^{-1/2}\ )\right).   
    \end{multline*}
\end{enumerate}
Therefore, the three sums in the lemma satisfy the following bounds, respectively
\begin{multline*}
    \ll \max_{\substack{n,m\leq Q}}\ \ \Big(\sum_{\begin{subarray}{c}n \sim N \\ m \sim M\end{subarray}}^{\sharp} d_{3}(n)d_{3}(m)c_{q}(n-m)f(n)g(m) \Big)\ \\  \ll\ \frac{Q^2}{q}+q^{11/8}Q+q^{9/16}Q^{1/2}.
\end{multline*}
 \begin{multline*}
        \ll \max_{\substack{n\leq Q\\Q<m\leq T}}\ \Big(\sum_{\begin{subarray}{c}n \sim N \\ m \sim M\end{subarray}}^{\sharp} d_{3}(n)d_{3}(m)c_{q}(n-m)f(n)g(m) \Big)\\ 
        \ll \ \Big( \frac{(Tx)^{K}}{q}+1\Big) \Big( \frac{QT^{1-d}}{q}+q^{11/8}Q^{1-d}+q^{9/16}Q^{1/2-d}\Big).\\
    \end{multline*}
\begin{multline*}
    \ll \max_{\substack{Q<n,m\leq T}}\Big(\sum_{\begin{subarray}{c}n \sim N \\ m \sim M\end{subarray}}^{\sharp} d_{3}(n)d_{3}(m)c_{q}(n-m)f(n)g(m) \Big)\\
    \ll\ \Big( \frac{(Tx)^{K}}{q}+1\Big)^2 \Big( \frac{T^{2-2d}}{q}+q^{11/8}Q^{1-2d}+q^{9/16}Q^{1/2-2d}\Big).
\end{multline*}
\end{proof}


\section{Proof of Theorem \ref{maintheorem}}\label{proof-of-theorem}
Take $q\leq x$, let $E_{a/q}(s),\  f_{a/q}(x),\  \Delta(a/q)$  be as in Theorem \ref{maintheorem}, and let $w(t),\ \hat{w}_q(n)$ be as in Lemmas \ref{lemma5} and \ref{lemma1}. Assume that all bounds can include a factor $x^{\varepsilon}$ which we do not write explicitly. Let $q\leq Y\leq x$ be a parameter and write
\begin{equation*}
    T:=\left(\frac{xq}{Y}\right)^{3}x^{\varepsilon-1}
\end{equation*}
so that choosing very large $j$ Lemma \ref{lemma5} says
\begin{equation}\label{tiny error}
    n>T\ \implies \ \hat{w}_q(n)\ \ll \ \frac{1}{x^{100}n^{4/3}}.
\end{equation}
From Theorem \ref{maintheorem} and Lemma \ref{lemma1}, we have
\begin{multline}
    \Delta(a/q)\ =\ \tilde{\Delta}(a/q)\  - \sum_{\substack{n\in (x-Y,x)\cup(2x,2x+Y)}}d_3(n)e\left(\frac{na}{q}\right)w(n)\\ 
    -\left[\underset{s=1}{\operatorname{Res}}\left\{E_{a/q}(s)\int_{0}^{\infty} w(t)t^{s-1}\ dt \right\}- \underset{s=1}{\operatorname{Res}}\left\{E_{a/q}(s)\frac{(2x)^s-x^s}{s} \right\}\right]\\
    =\ \sum_{n=1}^{\infty} d_3(n)e\left(\frac{na}{q}\right)w(n)\ -\sum_{\substack{n\in (x-Y,x)\cup(2x,2x+Y)}}d_3(n)e\left(\frac{na}{q}\right)w(n)\ +\ O\left( \frac{Y}{q}\right),\label{------}
\end{multline}
where the difference 
\begin{equation*}
    \frac{(2x)^s-x^s}{s} - \int_{0}^{\infty} w(t)t^{s-1} dt
\end{equation*}
is holomorphic and $\ll Y|x^{s-1}|$ and as in $\S 2$ of  \cite{Iv} the meromorphic part of $E_{a/q}(s)$ has coefficients $\ll 1/q$, which implies
\begin{equation*}
    \underset{s=1}{\operatorname{Res}}\left\{E_{a/q}(s)\left(\int_{0}^{\infty} w(t)t^{s-1}\ dt - \frac{(2x)^s-x^s}{s}\right) \right\}\ \ll\ \frac{Y}{q}.
\end{equation*}
Then equation (\ref{------}) gives
\begin{align}\nonumber
    \sum_{\substack{a=1\\ (a,q)=1}}^q |\Delta(a/q)|^2\ &\ll \ \sum_{\substack{a=1\\ (a,q)=1}}^q |\tilde{\Delta}(a/q)|^2+q \sum_{\substack{n,m\in (x-Y,x)\cup(2x,2x+Y)\\ n\equiv m\ (q)}}d_3(n) d_3(m)+\frac{Y^2}{q} \\
    &\ll\ \sum_{\substack{a=1\\ (a,q)=1}}^q |\tilde{\Delta}(a/q)|^2 + Y^2.
\end{align}
And we have
\begin{equation}\label{final touch}
    \sum_{\substack{a=1}}^q |\Delta(a/q)|^2 \ =\ \sum_{d|q} \sum_{\substack{a=1\\ (a,d)=1}}^d |\Delta(a/d)|^2 \ \ll\ \sum_{d|q}\left( \sum_{\substack{a=1\\ (a,d)=1}}^d |\tilde{\Delta}(a/d)|^2 + Y^2 \right).
\end{equation}
Write $N=x/d^3$ and $Q=d^3/2x$. We will use Lemma \ref{lemma4} and Lemma \ref{lemma5}, which say that for $n\leq Q$
\begin{equation}\label{bound for w(n)}
    \hat{w}_d(n)\ =\ \int_{0}^{\infty} w(t) P(\log((nt)^{1/3}/d))\ dt\ \ll\ x,
\end{equation}
and for  $Q\ll n \leq T$
\begin{align}\nonumber
    \hat{w}_d(n)\ &=\ \int_{0}^{\infty} w(t)\ \frac{e(6(Nt)^{1/3})}{(Nt)^{1/3}}\ dt\ + \underbrace{O\left(\int_{0}^{\infty} \frac{w(t)}{(Nt)^{(K+1)/3}}\ dt \right)}_{\text{say}\ \ :=\ L},\ \ \ (K\ \geq \ 1)\\ \nonumber
    &=\ \frac{1}{4\pi iN^{2/3}}\int_{0}^{\infty} w(t) t^{1/3} e'(6(Nt)^{1/3})\ dt\ +\ O\left(\int_{0}^{\infty} \frac{w(t)}{(Nt)^{(K+1)/3}}\ dt \right)\\ \nonumber
    &\leq \ \frac{1}{4\pi iN^{2/3}} \int_{0}^{\infty} \frac{d}{dt}\left\{w(t)t^{1/3}\right\} e(6(Nt)^{1/3})\ dt\ + O\left(\int_{0}^{\infty} \frac{w(t)}{(Nt)^{(K+1)/3}}\ dt \right)\\ \nonumber
    &=\ \underbrace{\frac{d^2}{4\pi i\ n^{2/3}} \int_{0}^{\infty}\frac{w(t)}{t^{2/3}}\ \ e(6(Nt)^{1/3})\ dt}_{\ll \dfrac{d^2x^{1/3}}{n^{2/3}}}\ +\ \underbrace{\frac{d^2}{4\pi i\ n^{2/3}} \int_{0}^{\infty}w'(t)t^{1/3}\ \ e(6(Nt)^{1/3})\ dt}_{\ll \dfrac{d^2x^{1/3}}{n^{2/3}}}\\ 
    &\ll\ \frac{d^2x^{1/3}}{n^{2/3}},\label{2nd bound for w(n)}
\end{align}
where $\int_{0}^{\infty} \dfrac{w(t)}{(Nt)^{(K+1)/3}}\ dt \ll x^{(2-K)/3}/N^{(K+1)/3} $, for $K\geq 1$. We will concentrate just on the main term of the last equality, the same argument clearly being applicable
for the terms in $L$. Write $\sum_{n,m}^\sharp$ for a sum with $n,m>d$ and $(nm,d)=1$. From the first bound of Lemma \ref{lemma9},
\begin{align}\label{1.}\nonumber
    &\frac{1}{d^3} \sum_{n,m\leq Q} d_3(n)d_3(m)c_d(n-m)\hat{w}_d(n)\overline{\hat{w}_d(m)}\\ \nonumber
    &=\ \frac{1}{d^3} \int_{0}^{\infty}\int_{0}^{\infty} w(t)w(t')\\ \nonumber
    &\hspace{0,5cm}\cdot \left(\sum_{n,m\leq Q}^{\sharp}  d_3(n)d_3(m)c_d(n-m)P(\log((nt)^{1/3}/d))P(\log((mt')^{1/3}/d)) \right)dtdt'\\ \nonumber
    &\ll\ \frac{1}{d^3} \int_{0}^{\infty}\int_{0}^{\infty} |w(t)|\ |w(t')|\left(\frac{Q^2}{d}+d^{11/8}Q+d^{9/16}\sqrt{Q}\right)\ dtdt'\\
    &\ll\ d^2\ +\ d^{11/8} x\ +\ d^{-15/16}x^{3/2}.
\end{align} 
From the second bound of Lemma \ref{lemma9},
\begin{align}\label{2.}\nonumber
    &\frac{1}{d^3} \sum_{\substack{n\leq Q\\ Q<m\leq T}} d_3(n)d_3(m)c_d(n-m)\hat{w}_d(n)\overline{\hat{w}_d(m)}\\ \nonumber
    &\leq\frac{1}{d} \int_{0}^{\infty}\int_{0}^{\infty} w(t)\frac{w(t')}{t'^{2/3}}\\ \nonumber
    &\hspace{0,5cm} \cdot \left(\sum_{\substack{n\leq Q\\Q < m \leq T}}^{\sharp} \frac{d_{3}(n)d_{3}(m)c_{q}(n-m)P(\log ((nt)^{1/3}/d))}{m^{2/3}} e\left(\frac{-6(mt')^{1/3}}{d}\right) \right)dtdt' +L_1\\ \nonumber
    &\ll  \underbrace{\frac{1}{d} \int_{0}^{\infty}\int_{0}^{\infty} w(t)\ \frac{w(t')}{t'^{2/3}} \frac{(Tx)^{1/3}}{d} \Big( \frac{QT^{1/3}}{d}+d^{11/8}Q^{1/3}+\frac{d^{9/16}}{Q^{1/6}}\Big)\ dt\ dt'}_{(1)}\\ \nonumber
    &+\underbrace{\frac{1}{d} \int_{0}^{\infty}\int_{0}^{\infty} w(t)\ w'(t')t'^{1/3} \frac{(Tx)^{1/3}}{d} \Big( \frac{QT^{1/3}}{d}+d^{11/8}Q^{1/3}+\frac{d^{9/16}}{Q^{1/6}}\Big)\ dt\ dt'}_{(2)}\\
    &\ll d^2\left(\frac{x}{Y}\right)^2\ +\ d^{11/8} \frac{x^2}{Y}\ +\ d^{-15/16}\frac{x^{5/2}}{Y},
\end{align}
where $(1)$ and $(2)$ give the same contribution and $L_1$ is the product of $(1/d^3) \sum d_3(n)d_3(m)c_d(n-m)\hat{w}_d(n)$ and $L$, which is defined above, for $\overline{\hat{w}_q(m)}$. From the third bound of Lemma \ref{lemma9},
\begin{align}\label{3.}\nonumber
     &\frac{1}{d^3} \sum_{Q < n,m \leq T} d_3(n)d_3(m)c_d(n-m)\hat{w}_d(n)\overline{\hat{w}_d(m)}\\ \nonumber
     &\leq\ \frac{d^4}{d^3}\int_{0}^{\infty}\int_{0}^{\infty} \ \frac{w(t)\ w(t')}{t^{2/3}\ t'^{2/3}}\\ \nonumber
     &\hspace{0,5cm} \cdot \left(\sum_{\sum_{Q < n,m \leq T}}^{\sharp} \frac{d_{3}(n)d_{3}(m)c_{q}(n-m)}{(nm)^{2/3}} e\left(\frac{6(nt)^{1/3}-6(mt')^{1/3}}{d}\right) \right)dtdt'+L_2\\ \nonumber
     &\ll d\int_{0}^{\infty}\int_{0}^{\infty} \ \frac{w(t)\ w(t')}{t^{2/3}\ t'^{2/3}}\ \frac{(Tx)^{2/3}}{d^2} \Big( \frac{T^{2/3}}{d}+\frac{d^{11/8}}{Q^{1/3}}+\frac{d^{9/16}}{Q^{5/6}}\Big)\ dt\ dt'\\ \nonumber
     &+ d\int_{0}^{\infty}\int_{0}^{\infty} \ w'(t)\ t^{1/3}\ w'(t')\ t'^{1/3}\ \frac{(Tx)^{2/3}}{d^2} \Big( \frac{T^{2/3}}{d}+\frac{d^{11/8}}{Q^{1/3}}+\frac{d^{9/16}}{Q^{5/6}}\Big)\ dt\ dt'\\
     &\ll d^2\left(\frac{x}{Y}\right)^4\ +\ d^{11/8}x\left(\frac{x}{Y}\right)^2\ +\ d^{-15/16}\frac{x^{23/6}}{Y^2},
\end{align}
where the term $L_2$ is treated similarly.  Therefore (\ref{bound for w(n)}) and (\ref{2nd bound for w(n)}) give
\begin{multline}\label{4.}
    \frac{1}{d^3} \sum_{\substack{n,m\ll T\\n\text{ or }m\leq d\\(nm,d)=1}}d_3(n)d_3(m)c_d(n-m) \hat{w}_d(n)\overline{\hat{w}_d(m)}\\
    \ll\ \frac{1}{d^3}\left(x^2\sum_{\substack{n\leq Q\\m\leq d}}\ +\ x^{4/3}d^2\sum_{\substack{Q<n\leq T\\m\leq d}} \frac{1}{n^{2/3}}\right)|c_d(n-m)|\\
    \ll \left(\frac{x^2Qd}{d^3}\ +\ \frac{x^{4/3}d^2d}{d^3}\int_{1}^{T}\frac{1}{t^{2/3}}\ dt\right)d \ll\ d\frac{x^2}{Y}.
\end{multline}
Combining all the bounds in the (\ref{1.}), (\ref{2.}), (\ref{3.}) and (\ref{4.}) we obtain
\begin{multline}\label{mainboundd}
     \frac{1}{d^3} \sum_{\substack{n,m\leq T\\(nm,d)=1}}d_3(n)d_3(m)c_d(n-m) \hat{w}_d(n)\overline{\hat{w}_d(m)}\\
     \ll\ d^2\left(\frac{x}{Y}\right)^4\ +\ d^{11/8}x\left(\frac{x}{Y}\right)^2\ =:\ \mathcal{E}.
\end{multline}
By Lemma \ref{lemma3} and (\ref{mainboundd}) we have
\begin{multline}\label{before final touch}
   < \frac{1}{d^3} \sum_{\substack{n,m\leq T\\(nm,d)=1}} \hat{w}_d(n)\overline{\hat{w}_d(m)}\left(\frac{1}{d^3}\sideset{}{'}\sum\limits_{h=1}^dA_{h/d}(n)\overline{A_{h/d}(m)}+\frac{1}{d^3}O(d^2)\right)\\
    = \underbrace{\frac{1}{d^6}\sum_{\substack{n,m\leq T\\(nm,d)=1}} | \hat{w}_d(n)\overline{\hat{w}_d(m)}|O(d^{2})}_{(1)}\\
    + \underbrace{\frac{1}{d^6} \sum_{\substack{n,m\leq T\\(nm,d)=1}} \hat{w}_d(n)\overline{\hat{w}_d(m)}\left(\sideset{}{'}\sum\limits_{h=1}^dA_{h/d}(n)\overline{A_{h/d}(m)}\right)}_{(2)}.
\end{multline}
From (\ref{bound for w(n)}) and (\ref{2nd bound for w(n)}), we can bound (1) in (\ref{before final touch}) as
\begin{multline}\label{lemma3 eki}
    \frac{1}{d^6}\sum_{\substack{n,m\leq T\\(nm,d)=1}} | \hat{w}_d(n)\overline{\hat{w}_d(m)}|O(d^{2})\ \ll\ \frac{1}{d^4}\left(x^2\sum_{\substack{n,m\leq Q}}\ +\ x^{2/3}d^4\sum_{\substack{n,m\leq T}} \frac{1}{(nm)^{2/3}}\right)\\
    \ll\ d^2\left(\frac{x}{Y}\right)^2\ll \mathcal{E},
\end{multline}
and from (\ref{mainboundd}) we bound (2) in (\ref{before final touch}) 
\begin{align}\label{latest touchss2}
     \frac{1}{d^6} \sum_{\substack{n,m\leq T\\(nm,d)=1}} \hat{w}_d(n)\overline{\hat{w}_d(m)}\left(\sideset{}{'}\sum\limits_{h=1}^dA_{h/d}(n)\overline{A_{h/d}(m)}\right)\ \ll \ \mathcal{E}.
\end{align}
Now we will consider the following sum with  $(nm,d)>1$ condition
\begin{equation*}
     \frac{1}{d^6} \sum_{\substack{n,m\leq T\\(nm,d)>1}} \hat{w}_d(n)\overline{\hat{w}_d(m)}\left(\sideset{}{'}\sum\limits_{h=1}^dA_{h/d}(n)\overline{A_{h/d}(m)}\right).
\end{equation*}
From Lemma \ref{lemma2}, we can bound the sum for the $(nm,d)>1$ as
\begin{equation*}
    \sideset{}{'}\sum\limits_{h=1}^dA_{h/d}(n)\overline{A_{h/d}(m)}\ \ll\ d(d^2,n)(d^2,m).
\end{equation*}
Therefore, (\ref{bound for w(n)}) and (\ref{2nd bound for w(n)}) give
\begin{multline*}
   \frac{1}{d^6} \sum_{\substack{n,m\leq T\\(nm,d)>1}} \hat{w}_d(n)\overline{\hat{w}_d(m)}\left(\sideset{}{'}\sum\limits_{h=1}^dA_{h/d}(n)\overline{A_{h/d}(m)}\right)\\
   \ll \ \frac{1}{d^5}\left(x^2\sum_{\substack{n,m\ll Q}}\ +\ x^{2/3}d^4\sum_{\substack{n,m\ll T\\m\leq d}} \frac{1}{(nm)^{2/3}}\right)(d^2,n)(d^2,m)
   \ll\ d\left(\frac{x}{Y}\right)^2,
\end{multline*}
and with the equation (\ref{latest touchss2}) we obtain
\begin{equation}\label{theoremde kullanılacak son}
     \frac{1}{d^6} \sum_{\substack{n,m\leq T}} \hat{w}_d(n)\overline{\hat{w}_d(m)}\left(\sideset{}{'}\sum\limits_{h=1}^dA_{h/d}(n)\overline{A_{h/d}(m)}\right)\ \ll \mathcal{E}.
\end{equation}
The final equality holds for $B_{h/d}$ as well, since the underlying lemmas are equally applicable to both terms. Therefore, Lemma \ref{lemma1}, (\ref{tiny error}) and (\ref{theoremde kullanılacak son}) we get
\begin{multline*}
    \frac{1}{\pi^{3/2}}\sum\limits_{h=1}^{d}|\tilde{\Delta}(h/d)|^2 \ =\ \frac{1}{d^6} \sum_{\substack{n,m\leq T}} \hat{w}_d(n)\overline{\hat{w}_d(m)}\left(\sideset{}{'}\sum\limits_{h=1}^dA_{h/d}(n)\overline{A_{h/d}(m)}\right)\\
    +\frac{1}{d^6} \sum_{\substack{n,m\leq T}} \hat{\nu}_d(n)\overline{\hat{\nu}_d(m)}\left(\sideset{}{'}\sum\limits_{h=1}^dB_{h/d}(n)\overline{B_{h/d}(m)}\right)
    +\ O\left(\frac{1}{x^{100}}\right)
    \ll\ \mathcal{E}. 
\end{multline*}
So (\ref{final touch}) gives 
\begin{equation*}
    \sum_{a=1}^{q}|\Delta(a/q)|^2\ \ll\ q^2\left(\frac{x}{Y}\right)^4\ +\ q^{11/8}x\left(\frac{x}{Y}\right)^2\ +\ Y^2
\end{equation*}
Optimizing the second and third error terms gives $ Y=x^{3/4}q^{11/32}$, and Theorem \ref{maintheorem} follows.

\section{Acknowledgment}
We are grateful to Hamza Ye\c{s}ilyurt, Ahmet M. Gülo\u{g}lu, and Tomos Parry for their interest and helpful feedback.


\begin{thebibliography}{Prime}
\bibitem[1]{Iv} Alexander Ivi\'c,\ \emph{On the Ternary Additive Divisor Problem and the Sixth Moment of the Zeta-Function; in Sieve Methods, Exponential Sums and Their Applications in Number Theory}, Cambridge Univ. Press, Cambridge, (1997) 1001-1039.

\bibitem[2]{Bl} Valentin Blomer,\ \emph{The average value of divisor sums in arithmetic progressions}, Q. J. Math., Vol\ \textbf{59} (2008), 275-286.

\bibitem[3]{BH}Bui, H.M., Heath-Brown, D.R., \emph{A note on the fourth moment of Dirichlet  $L$-functions}, Acta Arith. 141 Vol(\textbf{4}), (2010) 335-344.

\bibitem[4]{FI}J. B. Friedlander and H. Iwaniec, \emph{Incomplete Kloosterman sums and a divisor problem}, Ann. of Math. Vol(\textbf{2}) 121 (1985), no. 2, 319-350.

\bibitem[5]{HB} D. R. Heath-Brown, \emph{The divisor function $d_3(n)$ in arithmetic progressions}, Acta Arith. Vol\textbf{47} (1986), 29-56.

\bibitem[6]{FKM}\'E. Fouvry, E. Kowalski, and P. Michel, \emph{On the exponent of distribution of the ternary divisor function}, Mathematika \textbf{61} (2015), no. 1, 121-144.


\bibitem[7]{BHS}W.D. Banks, R. Heath-Brown and I.E. Shparlinski, \emph{On the average value of divisor sums in arithmetic progressions}, International Mathematics Research Notices 1 (2005).


\bibitem[8]{Ng} D. Nguyen, \emph{Topics in multiplicative number theory}, PhD thesis, University of California, Santa Barbara
(2021).



\bibitem[9]{PY} I. Petrow, M.P. Young,  \emph{The Weyl bound for Dirichlet L-functions of cube-free conductor}, Annals of Mathematics, Vol\ \textbf{192} (2020), 437-486.

\bibitem[10]{Pa1} Tomos Parry,\ \emph{A note on the distribution of $d_3(n)$ in arithmetic progressions}, The Ramanujan Journal, Vol\ \textbf{65} (2024), 1697-1708.

\bibitem[11]{Pa2}Tomos Parry, \emph{The distribution of $d_4(n)$ in arithmetic progressions}, accepted for publication in International Journal of Number Theory, (2026). 


\bibitem[12]{Mo}Y. Motohashi,  \emph{A note on the mean value of the zeta and $L$-functions},  Proc. Japan Acad. Ser. A
Math. Sci., Vol\textbf{61} (1985), 313-316.


\bibitem[13]{Mon}H. Montgomery - Topics in multiplicative number theory; Lecture Notes in Mathematics - Springer-Verlag,
Berlin-New York (1971).

\end{thebibliography}
\end{document}